\documentclass[12pt]{amsart}

\usepackage{xcolor}
\definecolor{nicered}{rgb}{0.6, 0, 0.1}
\definecolor{niceblue}{rgb}{0.06, 0.3, 0.57}
\definecolor{nicegreen}{rgb}{0.0, 0.51, 0.5}
\usepackage{color}

\usepackage[colorlinks=true,pagebackref,hyperindex,citecolor=nicegreen,linkcolor=niceblue,urlcolor=nicered]{hyperref}
\usepackage{amsmath,amsfonts,amssymb}
\usepackage{mathrsfs}
\usepackage{mathtools}
\usepackage{microtype}
\usepackage[all]{xy}
\usepackage[T1]{fontenc}
\usepackage{lmodern}
\usepackage{xparse}
\usepackage{enumitem}
\usepackage{comment}
\setlist[enumerate]{leftmargin=2em,label=\textup{(\roman*)}}
\usepackage{accents}
\newcommand\ubar[1]{%
  \underaccent{\bar}{#1}}
  
\usepackage{mathbbol}
\usepackage{stmaryrd}

\usepackage{cleveref}
\crefname{equation}{Eq.}{Eqs.}
\crefname{theorem}{Theorem}{Theorems} 
\crefname{lemma}{Lemma}{Lemmas}
\crefname{corollary}{Corollary}{Corollaries}
\crefname{proposition}{Proposition}{Propositions}
\crefname{definition}{Definition}{Definitions}
\crefname{remark}{Remark}{Remarks}
\crefname{example}{Example}{Examples}
\crefname{notation}{Notation}{Notations}
\crefname{setup}{Setup}{Setup}
\crefname{question}{Question}{Question}
\crefname{convention}{Convention}{Conventions}

\allowdisplaybreaks

\newtheorem{theorem}{Theorem}[section]
\newtheorem{corollary}[theorem]{Corollary}

\newtheorem{proposition}[theorem]{Proposition}

\theoremstyle{definition}
\newtheorem{definition}[theorem]{Definition}

\newtheorem{remark}[theorem]{Remark}

\numberwithin{equation}{section}

\newcommand{\fa}{\mathfrak{a}}

\newcommand{\CC}{\mathbb{C}}

\newcommand{\cJ}{\mathcal{J}}

\newcommand{\cO}{\mathcal{O}}

\newcommand{\cR}{\mathcal{R}}

\newcommand{\ZZ}{\mathbb{Z}}

\newcommand{\act}{\mathbin{\vcenter{\hbox{\scalebox{0.7}{$\bullet$}}}}}

\newcommand{\End}{\operatorname{End}}

\NewDocumentCommand \pd { m o }
{
\IfNoValueTF {#2}
{ \partial_{#1} }
{ \partial_{#1} \act #2 }
}
\NewDocumentCommand \fpd { m o }
{
\IfNoValueTF {#2}
{ \frac{\partial\ }{\partial #1 } }
{ \frac{\partial #2 }{\partial #1 } }
}

\newcommand{\R}{\mathbb{R}}

\newcommand{\J}{\mathcal{J}}

\NewDocumentCommand \seq { o m }
{
\IfNoValueTF {#1}
{ \ubar{#2} }
{ {#2}_1,\ldots,{#2}_{#1} }
}

\NewDocumentCommand \pt { o m }
{
\IfNoValueTF {#1}
{ #2 }
{ ({#2}_1,\ldots,{#2}_{#1}) }
}

\newcommand{\lambdab}{\pmb{\lambda}}
\newcommand{\lambdapb}{\pmb{\lambda'}}

\newcommand{\fab}{\pmb{\mathfrak{a}}}

\usepackage[textwidth=3.3 cm,textsize=small,shadow
]{todonotes}

\newcommand{\details}[2][]{} 
 
\title{A note on Bernstein-Sato ideals}

\author[J. \`Alvarez Montaner ]{Josep \`Alvarez Montaner}
\address{Departament de Matem\`atiques  and  Institut de Matem\`atiques de la UPC-BarcelonaTech (IMTech)\\  Universitat Polit\`ecnica de Catalunya } \email{josep.alvarez@upc.edu}

\thanks{Partially supported by grants  PID2019-103849GB-I00 (AEI/10.13039/501100011033) and 2017SGR-932 (AGAUR)}

\begin{document}

\begin{abstract}
We define the Bernstein-Sato ideal associated to a tuple of ideals and we relate it to the jumping points of the corresponding mixed multiplier ideals.

\end{abstract}

\maketitle

\setcounter{tocdepth}{1}


\section{Introduction}

Let $R$  be either the polynomial ring $\CC[x_1,\dots, x_n]$ over the complex numbers or the ring $\CC\{x_1,\dots, x_n\}$ of convergent power series in the neighbourhood of the origin, or any other point.  
The \emph{multiplier ideals} of a element$f$ or an ideal $\fa$ in $R$ are a family of nested ideals that play a prominent role in birational geometry (see Lazarsfeld's book \cite{Laz2004}). Associated to these ideals we have a set of invariants, the \emph{jumping numbers}, that are intimately related to other invariants of singularities. For instance, Ein, Lazarsfeld, Smith and Varolin \cite{ELSV2004} and independently Budur and Saito \cite{BudurSaito05}, proved that the negatives of the jumping numbers of $f$ in the interval $(0,1)$ are roots of the \emph{Bernstein-Sato polynomial} of $f$.  
Budur, Musta\c{t}\u{a} and Saito \cite{BMS2006a} extended the classical theory of Bernstein-Sato polynomials to the case of ideals and also proved that the jumping numbers of an ideal $\fa$ in the interval $(0,1)$ are roots of the \emph{Bernstein-Sato polynomial} of $\fa$.

\vskip 2mm

There is a natural extension of the theory of multiplier ideals to the context of tuples of germs $F:=f_1,\dots , f_\ell$ or tuples of ideals $\fab:=\fa_1,\dots,\fa_\ell $  in $R$. The main differences that we encounter in this setting is that, whereas the multiplier ideals come with the set of associated jumping numbers, the mixed multiplier ideals come with a set of \emph{jumping walls}. On the other side of the story we have the notion of \emph{Bernstein-Sato ideal} associated to a tuple of germs $F$ given by Sabbah \cite{Sabbah_ideal}. In the case of a tuple of plane curves,  Cassou-Nogu\`es and Libgober \cite{CNL11} related the Bernstein-Sato ideal with  the so-called  {\it faces of quasi-adjunction} which is a set of invariants equivalent to the jumping walls.

\vskip 2mm

The aim of this short note is to fill out the theory introducing  the notion of Bernstein-Sato ideal associated to a tuple ideals $\fab:=\fa_1,\dots,\fa_\ell $.  To such purpose we are going to follow the approach given by  Musta\c{t}\u{a} \cite{Mustata2019} where he relates the Bernstein-Sato polynomial of a single ideal $\fa=\big(f_{1},\dots , f_{r}\big)$  to the reduced Bernstein-Sato polynomial of $g= f_{1} y_{1} +\cdots  +  f_{ r} y_{r}$, where the $y_{j}$'s are new variables.
Finally, we show in Theorem \ref{main} that the negative of the jumping points of the mixed multiplier ideals of the tuple $\fab$ that are in the open ball of radius one centered at the origin belong to the zero locus of the Bernstein-Sato ideal of $\fab$.

\vskip 2mm

The theory of Bernstein-Sato polynomials and its relations with other invariants such as the multiplier ideals is vast and rich. In this note we tried to introduce only the essential concepts that we needed so we recommend those who are not that familiar with these topics to take a look at the 
surveys of Budur \cite{BudurNotes}, Granger \cite{Granger2010} or Jeffries, N\'u\~nez-Betancourt and the author \cite{BS_survey} for further insight. 

\vskip 2mm

{ {\bf Acknowledgements:} We would like to thank Guillem Blanco, Jack Jeffries and Luis N\'u\~nez-Betancourt for many helpful conversations regarding this work.

\section{Bernstein-Sato ideal of a tuple of ideals}\label{BS_ideals}


Let $R$  be either $\CC[x_1,\dots, x_n]$ or $\CC\{x_1,\dots, x_n\}$ and denote $\mathfrak{m}=(x_1,\dots, x_n)$ the (homogeneous)  maximal ideal.
Let $\fab:=\fa_1,\dots,\fa_\ell $ be a tuple of ideals   in $R$. For each ideal described by a set of generators   $\fa_i=\big(f_{i,1},\dots , f_{i ,r_i}\big)$
we consider $g_i= f_{i,1} y_{i,1} +\cdots  +  f_{i , r_i} y_{i , r_i}$ where the $y_{i,j}$'s are new variables. In particular we get a tuple  $G:= g_1,\dots, g_\ell$ in the ring  $A$ that will be either $\CC[x_1,\dots, x_n, y_{1,1},\dots, y_{\ell,r_\ell}]$ or $ \CC\{x_1,\dots, x_n, y_{1,1},\dots, y_{\ell,r_\ell}\}$. In the sequel, $d:=n+r_1 + \cdots + r_\ell$ will denote the number of variables in $A$.

\vskip 2mm

Associated to $R$ or $A$ we have the corresponding ring of  differential operators $$D_{R}=R \langle \partial_1, \dots , \partial_n  \rangle \hskip 2mm , \hskip 3mm  D_{A}=A \langle \partial_1, \dots , \partial_n , \partial_{1,1}, \dots ,\partial_{\ell,r_\ell} \rangle$$ where $\partial_i$ (resp. $\partial_{i,j}$) is the partial derivative with respect to $x_i$ (resp. $y_{i,j}$). That is, $D_{R}$ (resp. $D_{A}$) is the 
$\CC$-subalgebra of $\End_\CC(R)$ (resp. $\End_\CC(A)$)  generated by the ring and the partial derivatives.  
 
\begin{definition}
The \emph{Bernstein-Sato ideal} of the tuple $G$ is the ideal $B_G \subseteq \CC[s_1,\dots, s_\ell]$ generated by all the polynomials $b(s_1,\dots, s_\ell)$ satisfying the \emph{Bernstein-Sato functional equation}  
	\[ \delta(s_1,\dots, s_\ell) g_1^{s_1+1} \cdots g_\ell^{s_\ell+1}= b(s_1,\dots, s_\ell) g_1^{s_1} \cdots g_\ell^{s_\ell} \]
	where $\delta(s_1,\dots, s_\ell)\in D_A[s_1,\dots, s_\ell]$  and $b(s_1,\dots, s_\ell)\in \CC[s_1,\dots, s_\ell]$.
\end{definition}

Sabbah \cite{Sabbah_ideal} proved that $B_G \neq 0$ in the convergent power series case.  The proof of $B_G \neq 0$ in the polynomial ring case is completely analogous to the classical case of a single element. Indeed, it is enough to consider the local case.

\begin{remark}
Brian\c{c}on and Maisonobe showed in \cite{BriMai02} that
$$B_G^{\CC[x]} = \bigcap_{p \in \CC^d}  B_G^{\CC\{x-p\}}$$
where $B_G^{\CC[x]}$ denotes the Bernstein-Sato ideal of a tuple $G$ over the polynomial ring and $B_G^{\CC\{x-p\}}$ is the Bernstein-Sato ideal of $G$ in the convergent power series around  a point $p\in \CC^d$. 
\end{remark}

Now, since the  $g_i$ are pairwise without common factors, we have 
$$B_G \subseteq \Big( (s_1+1)\cdots (s_\ell+1)\Big).$$ 
(see \cite{May, BriMay} for details).


\begin{definition}
The \emph{reduced Bernstein-Sato ideal} of the tuple $G$ is the ideal $\widetilde{B}_G \subseteq \CC[s_1,\dots, s_\ell]$ generated by the polynomials $$\frac{b(s_1,\dots, s_\ell)}{(s_1+1)\cdots (s_\ell+1)}, $$ with $b(s_1,\dots, s_\ell) \in B_G$.

\end{definition}

Following the approach given by Musta\c{t}\u{a} \cite{Mustata2019} for the case of a single ideal, we consider the following:

\begin{definition}
Let $\fab =\fa_1,\dots,\fa_\ell $  be a tuple of ideals in $\cO_{X,O}$ and let $G:= g_1,\dots, g_\ell$ be its associated tuple of hypersurfaces.
We define the \emph{Bernstein-Sato ideal} of $\fab$ as 
$$ B_{\fab}:= \widetilde{B}_G \subseteq \CC[s_1,\dots, s_\ell]$$
\end{definition}

Our next result shows that $ B_{\fab}$ does not depend on the generators of the ideals $\fa_1,\dots,\fa_\ell $ and thus it is an invariant of the tuple $\fab$.

\begin{theorem}
Let $\fab:=\fa_1,\dots,\fa_\ell $ be a tuple of ideals  and, for each ideal, consider two different sets of generators 
$\fa_i=\big(f_{i,1},\dots , f_{i ,r_i}\big)$  and $\fa_i=\big(f'_{i,1},\dots , f'_{i ,s_i}\big)$. Consider the tuple $G=g_1, \dots , g_\ell$ with
$g_i= f_{i,1} y_{i,1} +\cdots  +  f_{i , r_i} y_{i , r_i}$ and the tuple $G'=g'_1, \dots , g'_\ell$
with $g'_i= f'_{i,1} y'_{i,1} +\cdots  +  f'_{i , s_i} y'_{i , s_i}$. 
Then $\widetilde{B}_G = \widetilde{B}_{G'}.$
\end{theorem}

\begin{proof}
Without loss of generality we may assume  that, for each ideal $\fa_i$, the set of generators $f'_{i,1},\dots , f'_{i ,s_i}$ is just  $f_{i,1},\dots , f_{i ,r_i}, h_i$ for a given $h_i\in \fa_i$.
Let $z_1,\ldots, z_{r_i}$ such that $h_i=z_1 f_{i,1}+\cdots+z_{r_i} f_{i,r_i}$. Then we have
 \begin{align*}
g_i' & = f_{i,1} y'_{i,1} +\cdots  +  f_{i , r_i} y'_{i , r_i} + h_i y'_{i , r_i +1}\\ &=
f_{i,1} y'_{i,1} +\cdots  +  f_{i , r_i} y'_{i , r_i} + (z_1 f_{i,1}+\cdots+z_{r_i} f_{i,r_i})  y'_{i , r_i +1}\\
& = f_1(y'_{i , 1}+z_1 y'_{i , r_i +1})+\cdots+f_\ell (y'_{i , r_i}+z_{r_i} y'_{i , r_i +1}).
\end{align*}
After a change of variables $y_{i,j}\mapsto y'_{i,j}+z_j y'_{i , r_i +1}$, this polynomial becomes $g_i$.
Since Bernstein-Sato ideals do not change by change of variables, we conclude that  $B_{G}=B_{G'}$ and the result follows.
\end{proof}

\section{Mixed multiplier ideals}\label{MMI}

Let  $\pi : X'  \longrightarrow X$ be a common log-resolution of a tuple of ideals $\fab=\fa_1,\dots,\fa_\ell $  in $R$.
Namely, $\pi$ is a birational morphism  such that 
\begin{enumerate}
\item[$\cdot$] $X'$ is smooth,
\item[$\cdot$]  $\fa_i\cdot\cO_{X'} = \cO_{X'}\left(-F_i\right)$ for some effective Cartier divisor $F_i$,  $i=1,\dots , \ell$,  
\item[$\cdot$] $\sum_{i=1}^\ell F_i+E$ is a divisor with simple normal crossings where $E = Exc\left(\pi\right)$ is the exceptional locus.
\end{enumerate}

The divisors $F_i=\sum_j e_{i,j} E_j$ are integral
divisors in $X'$ which can be decomposed  into their  {\it
exceptional} and {\it affine} part according to the support, i.e.
$F_i=F_i^{\rm exc} + F_i^{\rm aff}$ where
\[F_i^{\rm exc}= \sum_{j=1}^s e_{i,j} E_j \hskip 5mm {\rm and} \hskip 5mm  F_i^{\rm aff}= \sum_{j=s+1}^t e_{i,j} E_j.\]
Whenever $\fa_i$ is an $\mathfrak{m}$-primary ideal, the divisor $F_i$ is just supported on the exceptional
locus. i.e. $F_i=F_i^{\rm exc}$. We will also consider  the {\it relative canonical divisor} 
\[ K_{\pi}=\sum_{i=1}^s k_j E_j \]  which is a divisor in $X'$ supported on the
exceptional locus $E$ defined by the Jacobian determinant of the morphism $\pi$.

\begin{definition}\label{mmi}
The {\it mixed multiplier ideal} associated to a tuple  $\fab=\fa_1,\dots,\fa_\ell $ of ideals  in $R$ and a point
$\mathbf{\lambdab}=(\lambda_1,\dots,\lambda_\ell) \in \R_{\geqslant0}^\ell$ is defined
as 
\begin{equation*} 
\J(\fab^{\mathbf{\lambda}}):=\J(\fa_1^{\lambda_1}\cdots\fa_\ell^{\lambda_\ell})
= \pi_*\cO_{X'}\left(\left\lceil K_{\pi} - \lambda_1 F_1- \cdots -
\lambda_\ell F_\ell \right\rceil\right)
\end{equation*}
\end{definition}

In the classical case of a single ideal we have the notion of \emph{jumping numbers} associated to the sequence of multiplier ideals. 
The corresponding notion  in the context of mixed multiplier ideals is more involved.

\vskip 2mm

\begin{definition}\label{region}
Let $\fab=\fa_1,\dots,\fa_\ell $  be a tuple of ideals  in $R$. Then, for each $\lambdab\in \R_{\geqslant0}^\ell$, we
define:

\begin{itemize}
\item[$\cdot$] The {\it region} of $\lambdab $: \hskip 20mm
$\mathcal{R}_{\fab}\left(\lambdab\right)=\left\{\lambdapb \in
\R_{\geqslant 0}^\ell \hskip 2mm \left|
 \hskip 2mm \J(\fab^{\lambdapb})\supseteq\J(\fab^{\lambdab})\right\}\,\right. $.

\item[$\cdot$]  The {\it constancy region} of $\lambdab$: \hskip 3mm
$\mathcal{C}_{\fab}\left(\lambdab\right)=\left\{\lambdapb  \in
\R_{\geqslant 0}^\ell \hskip 2mm \left| \hskip 2mm
\J(\fab^{\lambdapb})=\J(\fab^{\lambdab})\right\}\,\right. $.
\end{itemize}

\end{definition}

The boundaries of these regions is where we have a strict inclusion of ideals. Therefore we may define:

\begin{definition}
Let $\fab=\fa_1,\dots,\fa_\ell $  be a tuple of ideals  in $R$. The {\it jumping wall} associated to $\lambdab \in
\R^\ell_{\geqslant 0}$ is the boundary of the region
$\cR_{\fab}(\lambdab)$. 
\end{definition}

In particular, we will be interested in the points of these jumping walls. In the sequel, $B_{\varepsilon}(\lambdab)$ stands for the open ball of radius $\varepsilon$ centered at 
a point $\lambdab \in \R^\ell$.

\begin{definition}\label{def:JP}
Let $\fab=\fa_1,\dots,\fa_\ell $  be a tuple of ideals  in $R$.  We say that $\lambdab \in \R^\ell_{\geqslant 0}$ is a
\emph{jumping point} of $\fab$ if
$\J(\fab^{\lambdapb})\varsupsetneq\J(\fab^{\lambdab})$
for all $\lambdapb \in \{\lambdab - \R_{\geqslant0}^\ell\} \cap
B_{\varepsilon}(\lambdab)$ and $\varepsilon >0$ small enough.
\end{definition}

From the definition of mixed multiplier ideals  we have that the
jumping points $\lambdab \in \R^\ell_{\geqslant 0}$   must lie on
hyperplanes of the form $H_j: \hskip 1mm e_{1,j} z_1+ \cdots +  e_{\ell,j} z_\ell= k_j +\nu_j $  for $j=1,\dots,s$ and $\nu_j \in \ZZ_{>0}$.

For $\lambda \in (0,1)$ we have $\J(\fa^{\lambda})= \J(g_\alpha^{\lambda})$ where $\fa=\big(f_{1},\dots , f_{r}\big)$ is a single ideal in $R$ and  $g_\alpha= \alpha_1 f_{1} +\cdots  + \alpha_r f_{ r} \in R$ with $\alpha_i \in \mathbb{C}$ is a general element (see \cite[Prop. 9.2.28]{Laz2004}).  As a consequence of a more general result of Musta\c{t}\u{a} and Popa given in \cite[Theorem 2.5]{MPIdeals} we also have  a relation between $\J(\fa^{\lambda})$ and  the multiplier ideal of the  associated hypersurface $g= f_{1} y_{1} +\cdots  +  f_{ r} y_{r}$ in  $A$.  

\vskip 2mm

\begin{definition} 
  Let  $\cJ =\big( Q_1(y), \dots, Q_s(y) \big)$ be an ideal in $A$. Then, ${\rm Coeff}(\cJ) \subseteq R$ is the ideal 
  generated by $$\{ Q_1(\alpha), \dots, Q_s(\alpha) \hskip 2mm | \hskip 2mm \alpha \in \mathbb{C}^r\}$$
  \end{definition}

  The result of Musta\c{t}\u{a} and Popa in the form that we need  is the following

\begin{proposition} \label{MP_mult}
Let $\fa=\big(f_{1},\dots , f_{r}\big)$ be an ideal in $R$ and let $g= f_{1} y_{1} +\cdots  +  f_{ r} y_{r}$ be the associated hypersurface in  $A$. Then, for any $\lambda \in \mathbb{Q} \cap (0,1)$ we have 
$$\J(\fa^{\lambda}) = {\rm Coeff}\big(\J(g^{\lambda})\big)$$
In particular,  the set of jumping numbers in the interval $(0,1)$ of $\fa$ and $g$ coincide.
\end{proposition}  

\vskip 2mm

The mixed multiplier ideals version of this result follows immediately from the following observation.

\begin{remark} \label{ray}
Consider a ray through the origin $L: (0,\dots, 0) + \mu(\alpha_1, \dots, \alpha_\ell)$ where the $\alpha_i$'s are positive integers. 
\begin{figure}[ht!!]
  \begin{center}
  \begin{minipage}{0.6\textwidth}
  \medskip
  \includegraphics[width=.7\textwidth]{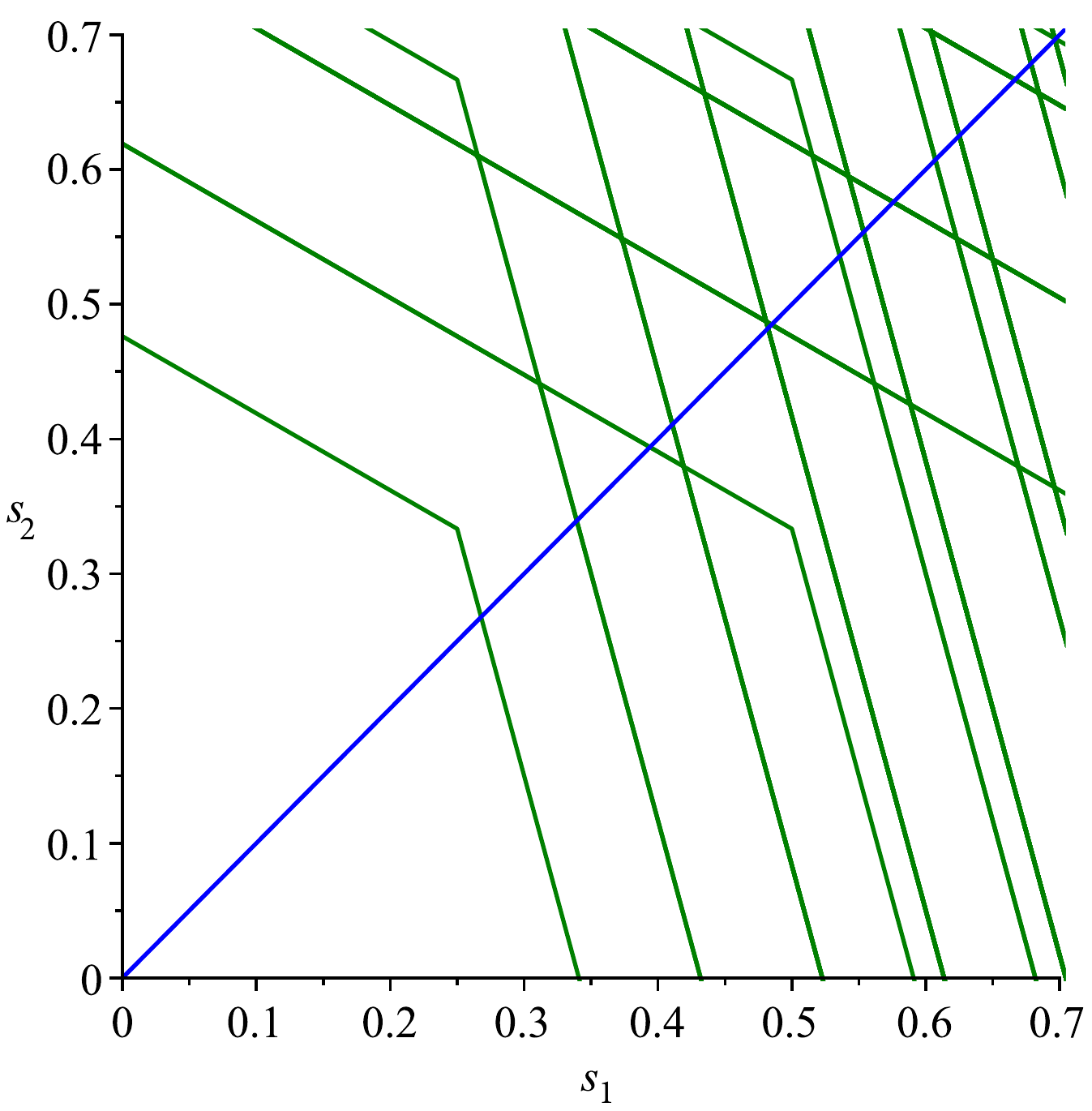}
  \label{fig:JW}
  \end{minipage}
  \end{center}
  \end{figure}
  
\noindent  Then, the jumping points of a tuple $\fab=\fa_1,\dots,\fa_\ell $  lying on $L$ are the jumping numbers of the ideal $ \fa_1^{\alpha_1} \cdots \fa_\ell^{\alpha_\ell} $
  \end{remark}

  \begin{corollary}
Let $\fab =\fa_1,\dots,\fa_\ell $  be a tuple of ideals in $R$, $G:= g_1,\dots, g_\ell$  its associated tuple of hypersurfaces and
consider $\lambdab \in \mathbb{Q}^\ell_{\geqslant 0}$ with Euclidean norm $\parallel \lambdab  \parallel <1$. Then, $\lambdab$ is a jumping point of $\fab$ if and only if it is a jumping point of $G$. 
  \end{corollary}

\begin{proof}
After Remark \ref{ray} we may assume that we have a single ideal $\fa=\big(f_{1},\dots , f_{r}\big)$ so its associated hypersurface is $g= f_{1} y_{1} +\cdots  +  f_{ r} y_{r}$. The result then follows from \ref{MP_mult}.
\end{proof}

In order to prove the main result of this section we will need the analytic definition of mixed multiplier ideal associated to a tuple $G= g_1,\dots, g_\ell$.

\begin{definition}
Let $G= g_1,\dots, g_\ell$ be a tuple in \(  A\). Let ${\overline{B}_{\epsilon}(O)} $ be a closed ball of radius $\varepsilon$ and center the origin $O\in \CC^d$. The mixed multiplier ideal (at the origin $O$) of \( G \) associated with  \( \lambdab \in \mathbb{Q}^\ell_{>0} \) is
\begin{equation*}
\mathcal{J}(g_1^{\lambda_1} \cdots g_\ell^{\lambda_\ell})_O= \big\{ h \in A \ \big|\ \exists\, \epsilon \ll 1\ \textnormal{such that}\ \int_{\overline{B}_{\epsilon}(O)} \frac{|h|^2}{|g_1|^{2\lambda_1}\cdots |g_\ell|^{2\lambda_\ell}} dx dy d\bar{x} d\bar{y} < \infty \big\}.
\end{equation*}

\end{definition}

\begin{remark}
As in the case of Bernstein-Sato ideals it is enough to consider this local case since we 
have $$\J(g_1^{\lambda_1} \cdots g_\ell^{\lambda_\ell}) = \cap_{p\in \CC^d} \J(g_1^{\lambda_1} \cdots g_\ell^{\lambda_\ell})_p.$$
If it is clear from the context we will omit the subscript referring to the point.
\end{remark}

\begin{theorem} \label{main}
Let $\fab=\fa_1,\dots,\fa_\ell $  be a tuple of ideals  in $R$. Let $\lambdab \in \mathbb{Q}^\ell_{\geqslant 0}$ be a
{jumping point} of $\fab$ with  Euclidean norm $\parallel \lambdab  \parallel <1$. Then $-\lambdab \in Z(B_{\fab})$.
\end{theorem}

\begin{proof}
Let $\lambdab \in \mathbb{Q}^\ell_{\geqslant 0}$ be a
{jumping point} of the tuple $G= g_1,\dots, g_\ell$ associated to $\fab$  with $\parallel \lambdab  \parallel <1$ and take $h\in \mathcal{J}(f^{\lambdab'}) \smallsetminus \mathcal{J}(f^{\lambdab})$ with $\lambdab' \in  \{\lambdab - \R_{\geqslant0}^\ell\} \cap
B_{\varepsilon}(\lambdab)$
 for $\varepsilon > 0$ small enough.  Therefore \[ \frac{|h|^2}{|g_1|^{2\lambda'_1}\cdots |g_\ell|^{2\lambda'_\ell}} \] is integrable but when we take the limit $\varepsilon \rightarrow 0$   we end up with \[ \frac{|h|^2}{|g_1|^{2\lambda_1}\cdots |g_\ell|^{2\lambda_\ell}} \] that is not integrable. Set $d=n+r_1 + \cdots + r_\ell$ and consider the complex zeta function
\begin{equation*} \label{int-eq}
 \int_{\mathbb{C}^d}|g_1|^{2s_1}  \cdots |g_\ell|^{2s_\ell} \varphi(x, y,\bar{x}, \bar{y}) dx dy d\bar{x} d\bar{y},
\end{equation*}
where $s_1,\dots, s_\ell$ are indeterminate variables  and \( \varphi(x, \bar{x}) \in C^\infty_c(\mathbb{C}^{d}) \) is a \emph{test function}, i.e. an infinitely many times differentiable function with compact support. Moreover $\varphi$ has holomorphic and antiholomorphic part. For any $b(s_1,\dots, s_\ell) \in B_G$  we have a Bernstein-Sato functional equation
	\[ \delta(s_1,\dots, s_\ell) g_1^{s_1+1} \cdots g_\ell^{s_\ell+1}= b(s_1,\dots, s_\ell) g_1^{s_1} \cdots g_\ell^{s_\ell} \]
Therefore
\begin{align*}
b^2(s_1,\dots, s_\ell) & \int_{\mathbb{C}^d}  \varphi(x, y,\bar{x}, \bar{y}) |g_1|^{2s_1}  \cdots |g_\ell|^{2s_\ell}dx dy d\bar{x} d\bar{y} = \\ = &  
 \int_{\mathbb{C}^d} \bar{\delta}^*\delta^*(s_1,\dots, s_\ell)\big(\varphi(x, y,\bar{x}, \bar{y})\big) |g_1|^{2(s_1+1)}  \cdots |g_\ell|^{2(s_\ell+1)} dx dy d\bar{x} d\bar{y}.
 \end{align*}
where $\bar{\delta}^*$ and $\delta^*$ denote the \emph{conjugate} and the \emph{adjoint} differential operators associated to $\delta$.
Notice that $|h|^2\varphi(x, \bar{x})$ is still a test function so
\begin{align*}
b^2(s_1,\dots, s_\ell) & \int_{\mathbb{C}^d}  |h|^2 \varphi(x, y,\bar{x}, \bar{y}) |g_1|^{2s_1}  \cdots |g_\ell|^{2s_\ell}dx dy d\bar{x} d\bar{y} = \\ = &  
 \int_{\mathbb{C}^d} \bar{\delta}^*\delta^*(s_1,\dots, s_\ell)\big(|h|^2\varphi(x, y,\bar{x}, \bar{y})\big) |g_1|^{2(s_1+1)}  \cdots |g_\ell|^{2(s_\ell+1)} dx dy d\bar{x} d\bar{y}.
 \end{align*}

\noindent Now we take a test function $\varphi$ which is zero outside the ball $\overline{B}_{\epsilon}(O)$ and identically one on a smaller ball  $\overline{B}_{\epsilon'}(O)\subseteq \overline{B}_{\epsilon}(O)$ and thus we get
\begin{align*}
b^2(s_1,\dots, s_\ell) & \int_{\overline{B}_{\epsilon'}(O)} |h|^2 |g_1|^{2s_1}  \cdots |g_\ell|^{2s_\ell} dx dy d\bar{x} d\bar{y}  = \\ = & 
 \int_{\overline{B}_{\epsilon'}(p)} \bar{\delta}^*\delta^*(s_1,\dots, s_\ell))\big(|h|^2\big)|g_1|^{2(s_1+1)}  \cdots |g_\ell|^{2(s_\ell +1)} dx dy d\bar{x} d\bar{y}.
 \end{align*}
Taking $s=-(\lambda'_1, \dots , \lambda'_\ell)$ we get
\begin{align*}
b^2(-\lambda'_1, \dots , -\lambda'_\ell)  &\int_{\overline{B}_{\epsilon'}(O)} \frac{|h|^2} {|g_1|^{2\lambda'_1}  \cdots |g_\ell|^{2\lambda'_\ell}} dx dy d\bar{x} d\bar{y}  =\\ = & 
 \int_{\overline{B}_{\epsilon'}(O)} \bar{\delta}^*\delta^*(-\lambda'_1, \dots , -\lambda'_\ell)\big(|h|^2\big) |g_1|^{2(1-\lambda'_1)}  \cdots |g_\ell|^{2(1-\lambda'_\ell)} dx dy d\bar{x} d\bar{y} \end{align*}
but the right-hand side is uniformly bounded for all $\varepsilon >0$. Thus we have 
\[
b^2(-\lambda'_1, \dots , -\lambda'_\ell)  \int_{\overline{B}_{\epsilon'}(O)} \frac{|h|^2} {|g_1|^{2\lambda'_1}  \cdots |g_\ell|^{2\lambda'_\ell}} dx dy d\bar{x} d\bar{y}   \leq  M < \infty\]
for some positive number $M$ that depends on $h$. Then, by the monotone convergence theorem we have to have \(b^2(-\lambda_1, \dots , -\lambda_\ell) = 0 \) and thus $-\lambdab \in Z(\tilde{B}_G)=Z(B_\fa)$.
\end{proof}

\bibliographystyle{alpha}
\bibliography{refs}

\end{document}